\documentclass[12pt,a4paper]{article}

\usepackage{amsmath,amsthm}
\usepackage{graphicx}

\def\Int{\operatorname{Int}}

\def\bord{\partial}

\def\gw{\operatorname{gw}}
\let\bydef\emph

\def\sbo{sub\-or\-bi\-fold}

\newtheorem{theo}{Theorem}[section]

\newtheorem{prop}[theo]{Proposition}
\newtheorem{lem}[theo]{Lemma}
\newtheorem{slem}[theo]{Sublemma}

\theoremstyle{definition}
\newtheorem*{defi}{Definition}

\newtheorem{rem}[theo]{Remark}

\newenvironment{rems}{\bigskip\emph{Remarks.}\begin{itemize}}{\end{itemize}}

\def\rr{\mathbf{R}}

\def\zz{\mathbf{Z}}
\def\nn{\mathbf{N}}

\def\calT{\mathcal{T}}

\def\calC{\mathcal{C}}
\def\oo{\mathcal{O}}
\def\so{\Sigma_{\oo_0}}

\title{One-ended $3$-manifolds without locally finite toric decompositions}
\author{Sylvain Maillot}

\begin{document}

\maketitle

\begin{abstract}
We introduce a class of one-ended open 3-manifolds which can be `recursively' defined from two compact 3-manifolds, and construct examples of manifolds in this class which fail to have a toric decomposition in the sense of Jaco-Shalen and Johannson.
\end{abstract}

\section{Introduction}

To start off, we introduce a class $\calC$ of open 3-manifolds which we view as a candidate for `the smallest class of open 3-manifolds for which the classification problem is interesting.' It is large enough so that exotic phenomena due to the topology at infinity can occur; yet it is small enough so that algorithmic problems---in particular the homeomorphism problem---make sense, and seem to have a decent chance of being decidable.

Throughout the paper we work in the PL category. Let $(X,Y,F_+,F_-,f,g)$ be a 6-tuple with the following properties:
\begin{itemize}
\item Both $X$ and $Y$ are connected, orientable, compact 3-manifolds.
\item $\bord X$ is connected.
\item $\bord Y$ has exactly two components, which are $F_+$ and $F_-$.
\item $f$ is an orientation-reversing homeomorphism from $\bord X$ to $F_-$.
\item $g$ is an orientation-preserving homeomorphism from $F_+$ to $F_-$.
\end{itemize}

To such a 6-tuple, we associate an open 3-manifold obtained by gluing together $X$ and an infinite sequence of copies of $Y$ using the homeomorphisms $f,g$. More precisely, we set
$$M(X,Y,F_+,F_-,f,g) = X \cup_{f_0} Y \times \{0\} \cup_{g_0} T \times \{1\} _{g_1} Y \times \{2\} \cup_{g_2} \cdots$$
where the gluing homeomorphisms are given by $f_0(x)=(f(x),0)$ for all $x\in X$, and $g_n(x,n)=(g(x),n+1)$ for all $x\in Y$ and $n\in\nn$. We denote by $\calC$ the class of 3-manifolds obtained in this way.

\begin{rems}
\item Every 3-manifold in $\calC$ is orientable, connected, and one-ended. Moreover, it has an exhaustion by compact submanifolds with connected boundary of fixed genus. We can call \bydef{genus} of a manifold $M\in\calC$ the minimal genus of $\bord X$ in a presentation of $M$ as $M(X,Y,F_+,F_-,g,f)$.
\item We can fix triangulations of $X$ and $Y$ with respect to which $f$ and $g$ are simplicial, so that algorithmic problems are well-defined. For the same reason, there are countably many manifolds in $\calC$ up to homeomorphism.
\item The class $\calC$ contains the original Whitehead manifold~\cite{whitehead:unity} as well as many other contractible 3-manifolds not homeomorphic to $\rr^3$. It does not contain those with infinite genus. Nor does it contain all Whitehead manifolds of genus~1, since there are  uncountably many of those, as shown by McMillan~\cite{mcmillan}.
\item One can define in a similar way a class $\calC_n$ for each dimension $n$, so that $\calC=\calC_3$. This construction is a special case of that of~\cite{maillot:homeosurf} of manifolds associated to a topological automaton. The automaton has two states corresponding to $X,Y$. Thus the classification for $\calC_2$ follows from the main theorem of~\cite{maillot:homeosurf}. The one-ended case, however, is much simpler than that; it is a straightforward consequence of the Kerekjarto classification theorem. Indeed, an orientable, one-ended surface is classified by its genus $h\in [0,+\infty]$, which is easily seen to be $+\infty$ if $Y$ has positive genus, and equal to the genus of $X$ otherwise.
\end{rems}

In this paper, we are interested in the question of which, among the exotic phenomena concerning open 3-manifolds, occur in the class $\calC$. We already mentioned that $\calC$ contains contractible manifolds which are not homeomorphic to $\rr^3$. For instance, the original Whitehead manifold has genus 1, the manifold $X$ being a solid torus, and $Y$ being the exterior of the Whitehead link in $S^3$.

Likewise, the manifold $M_1$ constructed by the author in~\cite{maillot:examples} is easily seen to belong to $\calC$. It has genus $1$, the manifold $X$ being again a solid torus, and $Y$ being the product of a circle by a compact orientable surface of genus $1$ with two boundary components. This manifold has the property that it is impossible to split it as a connected sum of prime manifolds, even allowing infinitely many factors, and allowing the factors to be noncompact. Note that the first such example was constructed by P.~Scott~\cite{scott:noncompact}. Scott's example has infinitely many ends (in fact, its space of ends is a Cantor set due to the treelike nature of the construction.) Thus it does not belong to $\calC$.

The manifold $M_3$ from~\cite{maillot:examples} does not have a locally finite splitting along 2-tori into submanifolds that are Seifert-fibered or atoroidal, thus showing that the theory of Jaco-Shalen~\cite{js:seifert} and Johannson~\cite{joh:hom} is difficult to extend to open 3-manifolds. This construction was inspired by Scott's work, so $M_3$ also has a Cantor set's worth of ends, and does not belong to the class $\calC$. The goal of this paper is to give an example of a manifold in $\calC$ with the same property.

In order to state the result, we recall some terminology. Let $M$ be an orientable 3-manifold. It is \bydef{irreducible} if every embedded 2-sphere in $M$ bounds a 3-ball. An embedded torus in $M$ is \bydef{incompressible} if it is $\pi_1$-injective. Following W.~Neumann and G.~Swarup~\cite{ns:canonical}, we call an embedded torus $T$ in $M$  \bydef{canonical} if it is incompressible and for every embedded, incompressible torus $T'\subset M$, there is an embedded torus $T''$ isotopic to $T'$ such that $T\cap T''=\emptyset$. 

\begin{defi}
Let $\calT=\{T_i\}_{i\in I}$ be a family of pairwise disjoint canonical tori in $M$. We say that $\calT$ is \bydef{complete} if every canonical torus in $M$ is isotopic to $T_i$ for some $i\in I$.
\end{defi}

\begin{theo}\label{thm:manifoldone}
There is an open 3-manifold $M_0$ in the class $\calC$ with the following properties.
\begin{enumerate}
\item $M_0$ is irreducible.
\item Every complete family of canonical tori in $M_0$ fails to be locally finite.
\end{enumerate}
\end{theo}

The key to constructing such a manifold is to ensure that there is a sequence $\{T_n\}$ of pairwise nonisotopic canonical tori in $M_0$ which fails to be locally finite no matter how the representatives of the various isotopy classes are chosen, because they are \bydef{trapped} by some compact subset $K\subset M_0$ which has to intersect all of them. In the paper~\cite{maillot:examples}, the $T_n$'s are pairwise nonhomologous, and separate different ends of the manifold $M_3$. By contrast, $M_0$ has only one end, and its $T_n$'s are all null-homologous, making it harder to prove that they are not isotopic.

In order to construct $M_0$, we first construct a \emph{$3$-orbifold} $\oo_0$ with similar properties, interpreted in the orbifold sense. It is irreducible, and contains a sequence $\{P_n\}$ of \bydef{pillows}, i.e.~spheres with four conical points of order 2, which is trapped by some compact set. The orbifold $\oo_0$ has underlying space $\rr^3$ and is designed so that it is easy to see that the $P_n$'s are pairwise nonisotopic. Then the manifold $M_0$ is constructed as a 2-fold cover of $\oo^3$. The $T_n$'s are taken to be the preimages of the $P_n$'s under the covering map.

The structure of the paper is as follows: in Section~\ref{sec:orbi} we construct the $3$-orbifold $\oo_0$ and prove its various properties. In Section~\ref{sec:man} we will prove Theorem~\ref{thm:manifoldone} by constructing the manifold $M_0$ and showing that it has the required properties.

\paragraph{Acknowledgements}
This work is partially supported by Agence Nationale de la Recherche through Grant~ANR-12-BS01-0004.

I would like to thank Gr\'egoire Montcouquiol, Michel Boileau and Luisa Paoluzzi for useful conversations.

\section{The orbifold case}\label{sec:orbi}

\subsection{Definition of the orbifold $\oo_0$}
Throughout the paper, we work in the PL or smooth category, and all manifolds and orbifolds are assumed to be connected and orientable. For terminology about 3-orbifolds, we refer to~\cite{bmp}.

Let $\oo$ be a $3$-orbifold. Two 2-\sbo s $F,F'$ are \bydef{isotopically disjoint} if there is a \sbo\ $F''$ isotopic to $F'$ and disjoint from $F$. An incompressible toric \sbo\  is \bydef{canonical} if it is isotopically disjoint from every incompressible toric 2-\sbo.

Let $K$ be a compact subset of $\oo$. A sequence $\{F_n\}$ of $2$-\sbo s is said to be \bydef{trapped by}  $K$ if no $2$-\sbo\ isotopic to any $F_n$ is disjoint from $K$. Thus, it is impossible to make $\{F_n\}$ locally finite by choosing different representatives of the various isotopy classes.

Let $B$ be a $3$-manifold homeomorphic to the $3$-ball and let $\alpha$ (resp.~$c$) be a properly embedded arc (resp.~circle) in $B$. Assume that $\alpha$ and $c$ are disjoint. Then we say that $\alpha \cup c$ is \emph{trivial} (resp.~a \bydef{Hopf tangle}) if $\alpha$ is trivial and $c$ bounds an embedded disk disjoint from $\alpha$ (resp.~meeting $\alpha$ transversally in a single point.)
To justify the terminology, note that if another $3$-ball $B'$ is glued to $B$ and $\alpha$ is extended to an unknot $c'$ in $B\cup B'$ in the obvious way, then $c\cup c'$ is an unlink (resp.~a Hopf link.)

We now come to our main construction: throughout the article we let $\oo_0$ be a $3$-orbifold with the following properties: its underlying space is Euclidean $3$-space. Its singular locus $\so$
consists in three unknotted, properly embedded lines $L_0$, $L_{12}$, and $L_{34}$, and four sequences of embedded circles $(c_i^{n})$ with $n\in\nn$ and $i\in \{1,2,3,4\}$. The general relative positions of the various components of $\so$ are shown on Figure~\ref{fig:general}. For each $n\in\nn$, $c_1^n \cup c_2^n \cup c_3^n \cup c_4^n$ is a four component link contained in a `box' $B_n$ whose intersection with $L_{12} \cup L_{34}$ consists of a trivial $4$-tangle $\alpha_1^n \cup \alpha_2^n \cup \alpha_3^n \cup \alpha_4^n$, with $\alpha_1^n \cup \alpha_2^n \subset L_{12}$ and
 $\alpha_3^n \cup \alpha_4^n \subset L_{34}$.

\begin{figure}[ht]
\begin{center}
\includegraphics{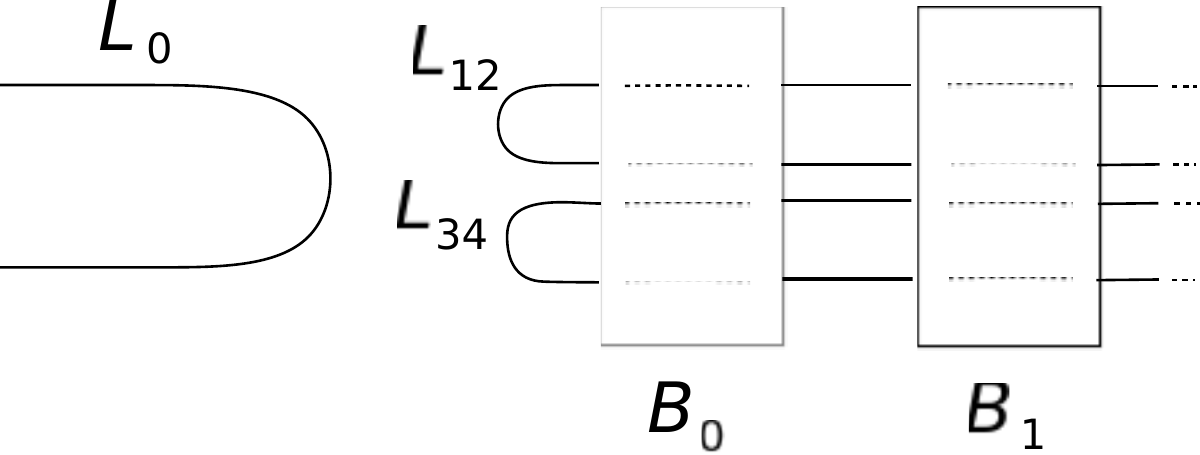}
\end{center}
\caption {\label{fig:general} General configuration of $\so$}
\end{figure}

Furthermore, we will assume that in each box, the $c_i^n$'s and the $\alpha_j^n$ satisfy the following properties:
\begin{enumerate}
\item each $c_i^n$ is unknotted;
\item $c_1^n\cup c_2^n$, $c_2^n\cup c_3^n$, $c_3^n\cup c_4^n$, and $c_4^n\cup c_1^n$ are Hopf links;
\item  $c_1^n\cup c_3^n$ and $c_2^n\cup c_4^n$ are unlinks;
\item for each $(i,j)\in \{1,2,3,4\}^2$, $c_i \cup \alpha_j$ is trivial if $i=j$, and a Hopf tangle otherwise.
\end{enumerate}

It is straightforward to see that such configurations exist, though they are far from being unique. Throughout the text, we assume that a choice has been made once and for all. Moreover, the choice is the same for all values of $n$, since we want the double manifold cover of $\oo_0$ to belong to the class $\calC$. 

What matters is not the precise nature of the content of the boxes, but rather the values of the linking numbers of the various components of $\so$; those are summarized by the graph depicted on~Figure~\ref{fig:linking}, with the convention that the linking number between two components is zero if there is no edge drawn between them, and one otherwise.

\bigskip

\begin{figure}[ht]
\begin{center}
\includegraphics{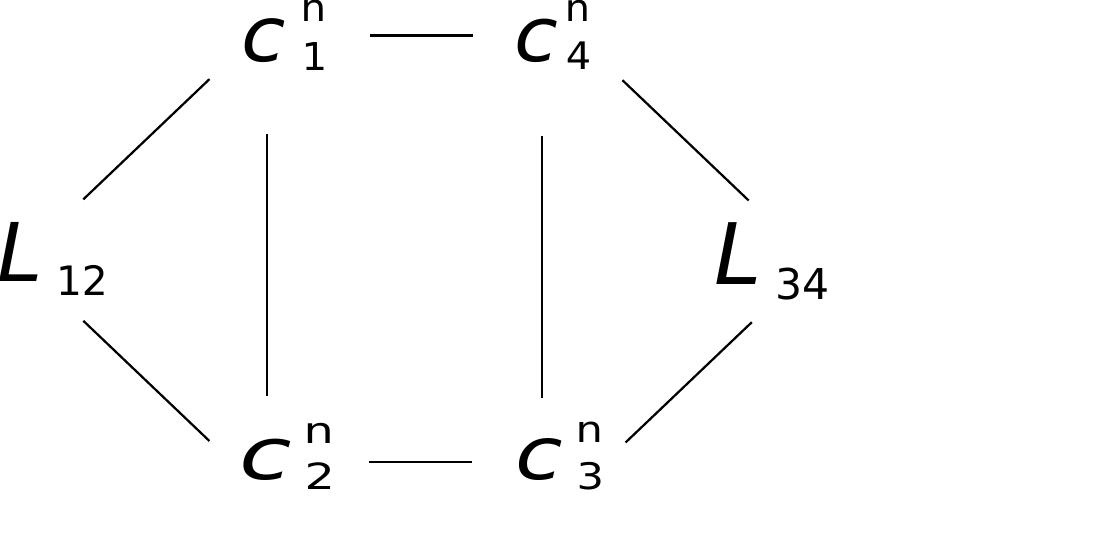}
\end{center}
\caption {\label{fig:linking} Linking numbers of components of $\so$}
\end{figure}

Finally, we assume that $\oo_0$ is a so-called $\pi$-orbifold, i.e.~each nontrivial local group is cyclic of order 2.

\begin{lem}\label{lem:ooirr}
$\oo$ is irreducible.
\end{lem}

\begin{proof}
Let $S\subset \oo$ be a $2$-\sbo\ of positive Euler characteristic. Then $|S|$ is a $2$-sphere intersecting $\so$ transversally. For homological reasons, the number of singular points of $S$ is even. Thus $S$ is either a nonsingular $2$-sphere or a football. Furthermore, by Alexander's Theorem, $|S|$ bounds a $3$-ball $B\subset\rr^3$.

Suppose that $S$ is nonsingular. If $B$ contains a singular point, then it contains some compact component $c_i^n$ of $\so$. Since $c_i^n$ is unknotted, it bounds a disk $D\subset\Int B$. Either $L_{12}$ or $L_{34}$ has linking number one with $c_i^n$. Thus for $L$ equal to $L_{12}$ or $L_{34}$ we have $D\cap L\neq\emptyset$, hence $B\cap L\neq\emptyset$. Since $L$ is noncompact, this contradicts the assumption that $S$ is nonsingular. Therefore, $B$ is nonsingular.

Suppose now that $S$ has two singular points. Those two points must belong to either $L_{12}$ or $L_{34}$. By symmetry we may assume it is $L_{12}$. Then $B\cap L_{12}$ is an unknotted arc. We need to show that $B\cap \so$ is in fact equal to this arc. Arguing by contradiction, assume that some $c_{i_0}^n$ is contained in $B$. Arguing as above using linking numbers, we can show that for every $i$, $c_i^n$ is contained in $\Int B$. Thus for the same reason $L_{34}$ intersects $B$, leading to a contradiction.
\end{proof}

\subsection{The pillows $P_n$ and their first properties}

Let us define the sequence of pillows $(P_n)$ and the compact subset $K'$ which traps them. As shown on Figure~\ref{fig:pillows} each $P_n$ meets $L_{12}$ and $L_{34}$ both twice; the $3$-ball bounded by $|P_n|$ contains a given box $B_m$ if and only if $m\le n$; finally, $K'$ is a $3$-ball meeting each of $L_0$, $L_{12}$, and $L_{34}$ in an unknotted arc, and for every $n$ the intersection of $K'$ with $P_n$ is a nonsingular disk.

For future reference, we let $\Pi$ denote the properly embedded plane shown on Figure~\ref{fig:pillows} and $H_0,H_1$ the closed half-spaces bounded by $\Pi$, so that $L_0\subset H_0$ and $H_1$ contains all other components of $\so$.

\begin{figure}[ht]
\begin{center}
\includegraphics{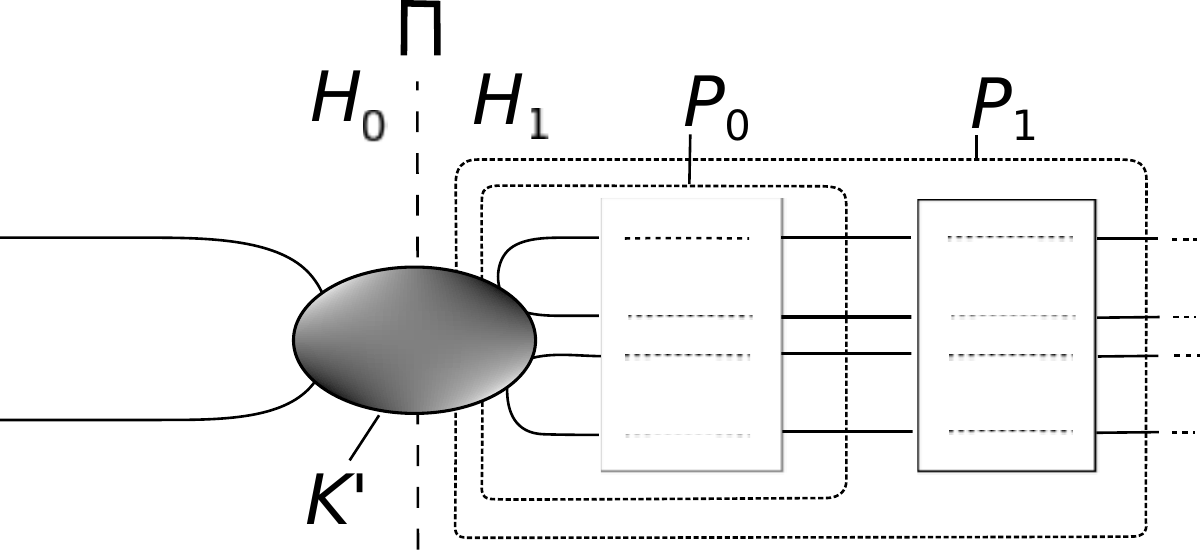}
\end{center}
\caption {\label{fig:pillows} The canonical pillows $P_n$, the compact set $K'$, and the plane $\Pi$}
\end{figure}

For future reference, for each $n$ we let $U_n$ denote the $3$-ball bounded by $|P_n|$ and set $X_n:=|\oo_0| \setminus \Int B_n$. Thus $X_n$ is noncompact and $\bord X_n=P_n$. 

We collect in the next proposition some facts about the $P_n$'s which are fairly easy to prove.
\begin{lem}\label{lem:first}
\begin{enumerate}
\item Each $P_n$ is incompressible.
\item The $P_n$'s are pairwise nonisotopic.
\item The $P_n$'s are trapped by $K'$.
\end{enumerate}
\end{lem}

\begin{proof}
Proof of~assertion~(i): If some $P_n$ were compressible, then $P_n$ would bound a solid pillow or be contained in a discal $3$-\sbo. Now $U_n$ contains at least four closed components of $\so$, so this is impossible.

Assertion~(ii) follows immediately from the fact that for $n\neq m$, the number of closed components of $\so$ contained in $U_n$ and $U_m$ are different.

Proof of~assertion~(iii): let $n$ be a natural number. Seeking a contradiction, we assume there is a pillow $P'_n$ isotopic to $P_n$ such that $P'_n\cap K'=\emptyset$. By Alexander's Theorem, the $2$-sphere $|P'_n|$ bounds a $3$-ball $U'_n$ in $\rr^3$. There are two cases:

\paragraph{Case 1} $K'$ is contained in $U'_n$. Then $|P'_n|$ is homologous to $\bord K'$ in $\rr^3\setminus \Int K'$. In this case, $P'_n$ must hit $L_0$, which contradicts our assumption that $P'_n$ is isotopic to $P_n$.

\paragraph{Case 2} $U'_n$ is disjoint from $K'$. Let $L_1,L_2$ (resp.~$L_3,L_4$) be the connected components of $\Int K'\cap L_{12}$ (resp.~$\Int K'\cap L_{34}$.) Again by homological reasons, and using the fact that $P'_n$ is isotopic to $P_n$, among the four singular points of $P'_n$, two must lie on some $L_i$ with $1\le i\le 2$, and the other two on some $L_j$ with $3\le j\le 4$. Now $\Int U'_n$ contains all $c_i^n$ for $1\le i\le 4$. As we have already seen several times, this leads to a contradiction.
\end{proof}

\begin{rem}\label{rem:irr}
The argument used to prove Lemma~\ref{lem:ooirr} shows that every $X_n$ is irreducible.
\end{rem}

\subsection{The $P_n$'s are canonical}

If $V$ is a solid torus and $V'$ a solid torus contained in $V$, then the \bydef{geometric winding number} of the pair $(V,V')$, denoted by $\gw (V,V')$, is the smallest natural number $n$ such that there is a meridian disk $D$ of $V$ with the property that $D\cap V'$ has $n$ connected components. A fundamental fact, due to Schubert, is the multiplicativity of $\gw$:

\begin{lem}[Schubert~\cite{schubert:vollringe}]\label{lem:schu}
Let $V,V',V''$ be solid tori such that $V\subset V'\subset V''$. Then $$\gw(V'',V) = \gw(V'',V') \cdot \gw(V',V).$$
\end{lem}

If $V$ is a solid torus and $c$ is an embedded circle in $V$, then the geometric winding number of the pair $(V,V')$ where $V'$ is a tubular neighbourhood of $c$ does not depend on the choice of $V'$. We shall denote this number by $\gw(V,c)$

Let $c,c'$ be two disjoint unknots in $\rr^3$. Choose two solid tori $V,V'$ such that $V\cap V'=\emptyset$ and $V$ (resp.~$V'$) is a tubular neighborhood of $c$ (resp.~$c'$.) Viewing $\rr^3$ as $S^3$ minus a point $\infty$, the complement of $\Int V$ in $S^3$ is also a solid torus, which we denote by $V_1$. By hypothesis, we have $V'\subset V_1$. We shall denote by $\gw (c^*,c')$ the geometric winding number of the pair $(V_1,V)$. Again, this does not depend on the various choices. If $c$ is an unknot in $\rr^3$ and $V$ is a solid torus contained in $\rr^3 \setminus c$, then $\gw(c^*,V)$ is defined similarly.

Note that if $c \cup c'$ is an unlink (resp.~a Hopf link), then $\gw(c^*,c')$ equals zero (resp.~one). Hence all geometric winding numbers involving compact components of $\so$ are determined by our construction.

By extension, we can define $\gw(L^*,c)$ when $L$ is an unknotted properly embeded line in $\rr^3$ and $c$ is an embedded circle in $\rr^3$ that misses $L$~: in the one-point compactification $S^3=\rr^3\cup \{\infty\}$, $L$ compactifies to a circle $L\cup\{\infty\}$. Let $V$ be the complement of an open tubular neighborhood of this circle such that $c\subset V$. Then we set $\gw(L^*,c):=\gw(V,c)$.
Again, all such numbers involving components of $\so$ are determined by the construction, and are equal to either zero or one.
Likewise, if $V'$ is a solid torus which misses $L$, then $\gw(L^*,V')$ is defined as $\gw(V,V')$ for a suitable choice of $V$. 

\begin{lem}\label{lem:solid}
Let $V$ be a solid torus in $|\oo_0|$ such that no point of $\bord V$ is singular. Then $\Int V$ contains at most one component of $\so$. If it does, then this component is compact and homotopic to the soul of $V$.
\end{lem}

\begin{proof}
Assume that $\Int V$ contains some component of $\so$. Then this component is compact. Up to changing the notation, we may assume it is $c_1^n$ for some $n$. For brevity, we drop the superscript $n$ in the sequel.

Observe that $c_1\subset V\subset S^3\setminus L_{12}$. Hence by Lemma~\ref{lem:schu}, we have $\gw(L_{12}^*,c_1) = \gw(L_{12}^*,V)\cdot
\gw(V,c_1)$. Since $\gw(L_{12}^*,c_1)=1$, we deduce $\gw(L_{12}^*,V)=\gw(V,c_1)=1$. We also have  $c_1\subset V\subset S^3\setminus L_{34}$. Since $\gw(L_{34}^*,c_1)=0$ and $\gw(V,c_1)=1$, Lemma~\ref{lem:schu} implies that $\gw(L_{34}^*,V)=0$.

Seeking a contradiction, assume that $c_3\subset V$. Then using the chain of inclusions $c_3\subset V\subset S^3\setminus L_{34}$ and the facts that $\gw(L_{34}^*,c_3)=1$ and $\gw(L_{34}^*,V)=0$ we get a contradiction. Hence $V\subset S^3\setminus c_3$. Using the chain $c_1\subset V\subset S^3\setminus c_3$ we obtain $\gw(c_3^*,V)=0$. The same argument applies, \emph{mutatis mutandis}, to $c_4$.

Again we prove by contradiction that $c_2\not\subset V$: otherwise using the chain $c_2\subset V\subset S^3\setminus c_3$ and the fact that $\gw(c_3^*,c_2)=1$, Lemma~\ref{lem:schu} would lead to a contradiction. 

By similar arguments, we exclude $c_3^m$ and $c_4^m$ for $m\neq n$ (using $L_{34}$), then $c_2^m$ (using $c_3^m$), and finally $c_1^m$ (using $c_4^m$.)
\end{proof}

\begin{lem}\label{lem:noannular}
For each $n\in\nn$, $X_n$ does not contain any essential, properly embedded annular $2$-\sbo.
\end{lem}

\begin{proof}
Arguing by contradiction, we let $A$ be such an annular $2$-\sbo. Thus $A$ is either a nonsingular annulus or a disk with two singular points.

\paragraph{Case 1} Suppose that $A$ is an annulus. An innermost disk argument shows that after an isotopy we may assume that $A\cap\Pi=\emptyset$. Then $\bord A$ is a union of two circles, which viewed on $P_n$ are parallel and have two singular points on each side. Let $Y$ be the manifold obtained from $|H_1\cap X_n|$ by removing small open tubular neighborhoods of $L_{12}$ and $L_{34}$. Then $Y$ has a manifold compactification diffeomorphic to a product $F\times [0,1]$, where $F$ is a sphere with four holes, and $F\times \{0\}$ is the image of a subset of $P_n$. Such a manifold does not contain any essential annulus with both boundary components in $F\times \{0\}$.

Since $A$ is incompressible in $Y$, we deduce that it is boundary-parallel in $Y$. Let $V$ be a parallelism region in $Y$ between $A$ and some annulus $A'$ in $P_n$. Let us show that no component of $\so$ lies in $V$. By Lemma~\ref{lem:solid}, $V$ contains at most one such component, say $c_1^m$ for some $m>n$. For simplicity of notation, set $c_i:=c_i^m$. Let $\gamma$ be a core of $A'$. Then we apply Lemma~\ref{lem:schu} to the chain $\gamma \subset V\subset S^3\setminus c_2$: from $\gw(V,\gamma)=1$ and $\gw(c_2^*,\gamma)=0$ we deduce $\gw(c_2^*,V)=0$. Then Lemma~\ref{lem:schu} applied to $c_1\subset V\subset S^3\setminus c_2$ yields a contradiction. This finishes the proof of Lemma~\ref{lem:noannular} in Case 1.

\paragraph{Case 2} Suppose that $A$ is a disk with two singular points. Again, we may assume that $A$ misses $\Pi$, and therefore lies in $H_1$. The intersection of $X_n$ with $L_{12} \cup L_{34}$ consists of four properly embedded half-lines, which we denote by $L_1,L_2,L_3,L_4$ in such a way that for each $m>n$ and each $1\le i\le 4$, $\alpha_i^m$ is contained in $L_i$.

By hypothesis, there exist two different indices $i_0,j_0$ in $\{1,2,3,4\}$ such that $A$ meets $L_i$ if and only if $i\in\{i_0,j_0\}$. Let $k_0,l_0$ be the two remaining indices. Let $Y$ be the $\pi$-orbifold whose underlying space is obtained from $|H_1 \cap X_n|$ by removing small open tubular neighborhoods of $L_{k_0}$ and $L_{l_0}$, and whose singular locus consists of $L_{i_0} \cup L_{j_0}$. Then $Y$ compactifies as $F\times [0,1]$, with $F$ an annulus (two-holed sphere) with two singular points, $F\times \{0\}$ corresponding to a subset of $P_n$.

As before, we deduce that $A$ is parallel in $Y$ to some disk $D\subset P_n$ with two singular points, which necessarily belong, one  to $L_{i_0}$ and the other to $L_{j_0}$. Let $B$ be the product region between $A$ and $D$, this time viewed as a \sbo\ of $\oo_0$. 
Observe that $B$ is a $3$-ball whose singular locus contains a trivial $2$-tangle consisting of the union of a subarc of $L_{i_0}$ and a subarc of $L_{j_0}$. Since $A$ is essential in $X_n$, $\Int B$ must in addition contain  $c_i^m$ for some $1\le i\le 4$ and $m>n$.

Arguing as in the proof of the irreducibility of $\oo$, we see that 
$\Int B$ contains $c_i^m$ for \emph{i}; in particular it contains $c_{k_0}^m$. We deduce that $c_{k_0}^m$, being an unknot in the ball $|B|$, bounds a disk $D'\subset |B|$. Such a disk must meet either $L_{k_0}$ or $L_{l_0}$. This contradiction completes the proof of
Case~2.
\end{proof}

At last we can prove the main result of this subsection:
\begin{prop}\label{prop:oocan}
Every $P_n$ is canonical.
\end{prop}

\begin{proof}
Assume that $P_n$ is not canonical. Let $P'_n$ be an incompressible toric $2$-\sbo\ of $\oo_0$ which cannot be isotoped off $P_n$. Since $\oo$ does not contain any turnover, $P'_n$ is a nonsingular torus or a pillow. Assume that $P'_n$ has been isotoped so as to intersect $P_n$ transversally and minimally.

Since both $P_n$ and $P'_n$ are incompressible, a component of $P_n\cap P'_n$ is essential on $P_n$ if and only if it is essential on $P'_n$. Using the irreducibility of $\oo_0$ and a standard argument, one can remove any inessential intersection component between $P_n$ and $P'_n$. Hence, by our assumption of minimality, each component of $P_n\cap P'_n$ is essential.
As a consequence, some connected component $A$ of $X_n\cap P'_n$ is an annular $2$-\sbo, which is essential as a \sbo\ of $X_n$. This contradicts Lemma~\ref{lem:noannular}.
\end{proof}

\section{The manifold $M_0$}\label{sec:man}
In this section, we let $p:M_0\to\oo_0$ be the only two-fold manifold cover of $\oo_0$. In other words, $M_0$ is the two-fold branched cover over $\rr^3$ with branching locus $\so$ and $p$ is the canonical projection. As in the compact case, this cover can be constructed by splitting $\oo_0$ along a Seifert surface for $\so$, taking two copies of the resulting manifold, and gluing them appropriately. Alternatively, the orbifold fundamental group $\pi_1\oo_0$ is (infinitely) generated by \emph{meridians}, i.e.~small circles around each component of $\so$. All meridians have order~2. There is an index~2 normal subgroup $\Gamma$ of $\pi_1\oo_0$ which is the kernel of a homomorphism from $\pi_1\oo_0$ to $\zz/2\zz$ sending all meridians to the nontrivial element. This group is torsion~free, and $p:M_0\to\oo_0$ is the corresponding manifold cover. We denote by $\tau:M_0\to M_0$ the involution such that $\oo_0=M_0/\tau$.

By construction, $\oo_0$ is obtained by gluing $K'$ and countably many copies of some compact 3-orbifold $Y'$ with two boundary components, glued to each other `in the same way'. Moreover, the preimage of $K'$ by $p$ is connected and has connected boundary, and the preimage of each copy of $Y'$ is connected and has exactly two boundary components. It follows that $M_0$ belongs to the class $\mathcal C$ defined in the introduction.

For each $n\ge 0$, set $T_n:=p^{-1}(P_n)$ and $Y_n:=p^{-1}(X_n)$. Thus each $T_n$ is an embedded $2$-torus in $M$. Set $K:=p^{-1}(K')$. Then $K$ is a compact subset of $M$. We need to prove that $M$ has all the properties stated in Theorem~\ref{thm:manifoldone}. 

Since $\oo_0$ is irreducible, $M_0$ is irreducible~(\cite[Theorem~3.23]{bmp}.) For the same reason, Remark~\ref{rem:irr} implies that $Y_n$ is irreducible for every $n$. By the equivariant Dehn Lemma, incompressibility of every $T_n$ follows from Lemma~\ref{lem:first}(i).

\begin{lem}
The $T_n$'s are pairwise nonisotopic.
\end{lem}

\begin{proof}
Assume that there exist $n,m$ with $n\neq m$ such that $T_m$ is isotopic to $T_n$. By~\cite[Prop~4.5]{bz:link}, there is a $\tau$-equivariant isotopy. Hence $P_m$ is isotopic to $P_n$, contradicting Lemma~\ref{lem:first}(ii).
\end{proof}

Our next goal is to show that every $T_n$ is canonical. For this we need a lemma.

\begin{lem}\label{lem:noannman}
There is no properly embedded essential annulus in $Y_n$.
\end{lem}

\begin{proof}
Put a $\tau$-invariant riemannian metric on $Y_n$ with mean-convex boundary. Arguing by contradiction, we assume that $Y_n$ contains a properly embedded essential annulus. Let $A$ have least area among such annuli. Set $A':=\tau(A)$. If $A\cap A'=\emptyset$ or $A=A'$, then $p(A)$ is an annular $2$-\sbo\ in $X_n$. Thus by Lemma~\ref{lem:noannular}, we deduce that $p(A)$ is inessential. This implies that $A$ is inessential, contradicting our hypothesis.

Hence by the Meeks-Yau trick we may assume that $A$ and $A'$ intersect transversally in a nonempty disjoint union of curves and arcs. Since both $A$ and $A'$ are essential, a curve or arc is essential on $A$ if and only if it is esssential on $A'$.

If there exists an inessential curve in $A\cap A'$, we let $D$ be a disk of minimal area on $A$ or $A'$ bounded by such a curve. Note that $D$ is automatically innermost. Thus a classical exchange/roundoff argument leads to a contradiction (the disk exchange can be realized by an isotopy because $Y_n$ is irreducible.) 

If there is an inessential arc in $A\cap A'$, we can argue similarly using a disk of minimal area cobounded by such an arc and an arc in $\bord Y_n$. In order to ensure that the disk exchange can be realized by an isotopy, we use the irreducibility of $Y_n$ and the incompressibility of $\bord Y_n$.

Suppose there is an essential curve in $A\cap A'$. Let $\gamma$ be such a curve. Pick a basepoint on $\gamma$ and consider the covering space $Y_\gamma$ of $M$ whose fundamental group is the infinite cyclic group generated by $\gamma$. Let $\tilde A$ and $\tilde A'$ be lifts to $Y_\gamma$ of $A$ and $A'$ respectively. Observe that the inclusions of $\tilde A$ and $\tilde A'$ in $Y_\gamma$ are homotopy equivalences. In particular,  $\tilde A$ and $\tilde A'$ are incompressible.

\begin{slem}\label{slem:boundary}
\begin{enumerate}
\item The two boundary components of $\tilde A$ lie on distinct components of $\bord Y_\gamma$.
\item The same property holds for $\tilde A'$.
\end{enumerate}
\end{slem}

\begin{proof}
It suffices to prove the first point. We argue by contradiction, letting $F$ be a component of $\bord Y_\gamma$ containing all of $\bord \tilde A$. Observe that $F$ is an open annulus whose inclusion in $Y_\gamma$ is a homotopy equivalence. Hence there is an annulus $A''$ contained in $F$ and such that $\bord A''=\bord \tilde A$. Thus $T:=\tilde A \cup A''$ is an embedded torus. Since $\pi_1(Y_\gamma)$ is cyclic, $T$ is compressible. As $T$ contains an incompressible annulus, it does not lie in a $3$-ball. Therefore, $T$ bounds a solid torus $V$, which is a parallelism region between $\tilde A$ and $A''$. In particular, $\tilde A$ is isotopic to $A''$ relative to the boundary. Projecting such an isotopy to $Y_n$ leads to a contradiction with the fact that $A$ is essential. This completes the proof of Sublemma~\ref{slem:boundary}.
\end{proof}

Note that in the previous discussion, $\gamma$ was an arbitrary essential curve of intersection between $A$ and $A'$. Now we make a more specific choice of $\gamma$: we assume that some annulus cobounded by $\gamma$ and some boundary component of $A$ or $A'$ has least possible area. Without loss of generality, we assume that such an annulus, called $A_1$, is contained in $A$. We let $A_2$ be the subannulus of $A$ whose boundary is the union of $\gamma$ and the other component of $\bord A$, and $A_3,A_4$ be the annuli such that $A_3\cap A_4=\gamma$ and $A_3\cup A_4=A'$.

We set $A_5:=A_1\cup A_3$ and $A_6:=A_1\cup A_4$. Then $A_5$ and $A_6$ are properly embedded annuli in $X_n$. As before we consider lifts $\tilde A_5$ and $\tilde A_6$ of those annuli to $Y_\gamma$. By the second assertion of Sublemma~\ref{slem:boundary}, at least one of $\tilde A_5$ and $\tilde A_6$ has its two boundary components on different components of $\bord Y_\gamma$. Assume $\tilde A_5$ has this property. Then $A_5$ is essential: otherwise one could lift a boundary-parallelism region for $A_5$ to $M_\gamma$ and obtain one for $\tilde A_5$. Now by choice of $A_1$, the area of $A_5$ is less than or equal to that of $A$, and by rounding the corner, we obtain an essential annulus which contradicts the minimizing property of $A$.

Thus we are left with the case where $A\cap A'$ contains an essential arc.

\begin{slem}
There are at least two distinct essential arcs in $A\cap A'$.
\end{slem}

\begin{proof}
Let $\gamma_1,\gamma_2$ denote the boundary components of $A$ and $\gamma_3,\gamma_4$ denote those of $A'$. Thus $\gamma_1$ (resp.~$\gamma_3$) is isotopic to $\gamma_2$ (resp.~$\gamma_4$) on $T_n$. If $A\cap A'$ consists of only one essential arc $\alpha$, then there are only two points in $\bord A\cap \bord A'$. If $\gamma$ is not isotopic to $\gamma_3$, then $\gamma_1\cap \gamma_3$ is nonempty, and so are $\gamma_1\cap\gamma_4$, $\gamma_2\cap\gamma_3$, and $\gamma_2\cap\gamma_4$, resulting in a contradiction.

Hence the $\gamma_i$'s are all pairwise isotopic, and the two intersection points belong to the same two arcs, which means that $\alpha$ is inessential, again a contradiction.
\end{proof}

Choose as basepoint an endpoint of some arc in $A\cap A'$ and let $Y_T$ be the covering space of $Y_n$ such that $\pi_1(Y_T)=\pi_1(T_n)$. Let $U_n$ be a torus component of $\bord Y_T$ whose projection is $T_n$, the said projection being a homeomorphism. Let $\tilde A$ and $\tilde A'$ be lifts to $Y_T$ of $A$ and $A'$ respectively, which both meet $U_n$.

\begin{slem}
\begin{enumerate}
\item The two components of $\bord\tilde A$ lie on distinct components of $\bord Y_T$.
\item The same holds for $\tilde A'$.
\end{enumerate}
\end{slem}

\begin{proof}
Again we only have to prove the first statement. Arguing by contradiction, assume both components of $\bord\tilde A$ belong to the same component of $\bord Y_T$. By construction, this component is $U_n$. Then there is an annulus $A''\subset U_n$ such that $\bord A''=\bord \tilde A$. Set $T=A''\cup \tilde A$. Then $T$ is an embedded torus in $Y_T$. By pushing $T$ slightly we get $T'\subset \Int Y_T$.

Suppose that $T$ (or equivalently $T'$) is compressible. Since $A''$ is not contained in any $3$-ball, neither is $T$, so $T$ bounds a solid torus $V$. Since the core of $A''$ is a primitive element of $U_n$, $V$ is a parallelism region between $A''$ and $\tilde A$. This shows that $\tilde A$ is inessential, and so is $A$, a contradiction.

Hence $T$ is incompressible. Note that $H_2(Y_T)$ is isomorphic to $\zz$ and generated by the class of $U_n$. Thus $T$ is either null-homologous or homologous to $U_n$.

If $T$ is null-homologous, then $T$ bounds some compact submanifold $V$. Since $T$ is incompressible in $Y_T$, it is also incompressible in $V$, and by van Kampen's theorem, the induced homomorphisms $\pi_1 T\to \pi_1V \to \pi_1 Y_T$ are both injective. It follows that $\pi_1V$ is isomorphic to $\zz^2$, which is impossible with $\bord V$ connected.

Thus $T$ is homologous to $U_n$, and so is $T'$. Let $X$ be a compact submanifold whose boundary is $U_n\cup T'$. Arguing as before, we see that $\pi_1X$ is isomorphic to $\zz^2$. Since $X$ is irreducible, $X$ is homeomorphic to $T^2\times I$. Again this implies that $\tilde A$ is boundary-parallel.
\end{proof}

We can now finish the proof of Lemma~\ref{lem:noannman}. Since $Y_T$ covers a noncompact manifold, it is itself noncompact. In particular, $Y_T$ is not homeomorphic to $T^2\times I$. Thus, homology considerations as above show that $U_n$ is the only torus in $\bord Y_T$. It follows that the boundary components $\gamma,\gamma'$ of $\tilde A$ and $\tilde A'$ that do not lie on $U_n$ are contained in some annulus component $A''$ of $\bord Y_T$, and are freely homotopic (and noncontractible) there. Hence there is a bigon between $\gamma$ and $\gamma'$, and we can use an exchange/roundoff argument to get a contradiction. This completes the proof of Lemma~\ref{lem:noannman}.
\end{proof}

The fact that the $T_n$'s are canonical now follow as Proposition~\ref{prop:oocan} followed from Lemma~\ref{lem:noannular}. Hence to prove Theorem~\ref{thm:manifoldone}, we only need to show the following statement:

\begin{lem}\label{lem:trapped}
The compact subset $K=p^{-1}(K')\subset M$ traps each $T_n$.
\end{lem}

\begin{proof}
Arguing by contradiction, we assume that for some $n$ there exists an embedded torus $T'\subset M\setminus K$ which is isotopic to $T_n$. Our aim is to show that after isotopy, $T'$ can be chosen to be invariant under $\tau$. Then by~\cite[Proposition~4.5]{bz:link}  $T'$ and $T_n$ are equivariantly isotopic. This implies that the image of $T'$ is an embedded pillow in $\oo\setminus K'$ isotopic to $P_n$, which contradicts the fact that $P_n$ is trapped by $K'$.

Let $F$ be a parallel copy of $\bord K$ lying outside $K$, and $X$ be the (compact) product submanifold whose boundary is the disjoint union of $\bord K$ and $F$. Thus $X\cap K=\bord K$. Without loss of generality we may assume that $T'\cap X=\emptyset$. We are going to use the Jaco-Rubinstein theory of PL minimal surfaces, as modified in~\cite{maillot:center}. Choose a triangulation $\calT$ of $M$ which is in general position with respect to $T'$ and invariant under $\tau$. Let $w_0$ be the weight of $T'$ with respect to $\calT$. By repeatedly subdividing $\calT$ on $X$, with may assume that any normal surface which meets both components of $\bord X$ has weight strictly greater than $w_0$, keeping the invariance under $\tau$. Then choose a $\tau$-invariant regular Jaco-Rubinstein metric in the sense of~\cite{maillot:center} so that the PL area is well-defined.

Let $T''$ be a normal torus of least PL area in the isotopy class of $T'$. We claim that $T''$ misses $K$. Indeed, $K$, being the double branched cover of a 3-ball with branching locus a trivial tangle, is a handlebody. The same is true for $K\cup X$. Since $T''$ is an incompressible torus, it cannot be contained in $K\subset X$. Hence if it met $K$, it would meet both boundary components of $X$, contradicting the fact that its weight is at most $w_0$.

Since the Jaco-Rubinstein metric we use is $\tau$-invariant, $\tau T''$ also has least PL area in its isotopy class. Now $T''$ is isotopic to $T_n$, which is $\tau$-invariant, so $\tau T''$ is isotopic to $T''$.
By Jaco-Rubinstein's version~\cite{jr:min} of a Theorem of Freedman-Hass-Scott, it follows that $T''$ and $\tau T''$ are disjoint or equal, i.e. $T''$ is invariant or equivariant. We still have to rule out the latter possibility.

Suppose that $T''\cap \tau T''=\emptyset$. Since $T''$ and $\tau T''$ are isotopic, they are parallel. Let $Y$ be the compact submanifold bounded by $T''\cup \tau T''$ and diffeomorphic to $T^2\times I$. Note that $Y$ is unique because $M$ is noncompact. Hence $Y$ is $\tau$-invariant. By the classification of involutions of $T^2\times I$, noting that $\tau$ exchanges the two boundary components, this implies that $\Int X$ contains an invariant embedded torus whose projection in $\oo$ is a one-sided, nonorientable $2$-suborbifold. Since $\oo$ does not contain any such suborbifold, we have reached the desired contradiction. Hence the proof of Lemma~\ref{lem:trapped} is complete.
\end{proof}

\bibliographystyle{alpha}
\bibliography{ex}

\noindent Institut Montpelli\'erain Alexander Grothendieck\\ CNRS, Universit\'e de Montpellier\\ 
\texttt{sylvain.maillot@umontpellier.fr}

\end{document}